\documentclass[11pt]{article}

\usepackage{amssymb,amsmath,amsfonts,amsthm}
\usepackage{latexsym}
\usepackage{graphics}
\usepackage{indentfirst}
\usepackage{hyperref}
\usepackage[capitalise, noabbrev, nameinlink]{cleveref}
\usepackage{comment}
\usepackage[shortlabels]{enumitem}
\usepackage[english]{babel}
\usepackage{tikz}
\usetikzlibrary{fit,arrows.meta,backgrounds}
\usetikzlibrary{shapes,decorations,graphs,graphs.standard,quotes,backgrounds}
\usetikzlibrary{decorations.markings}
\usepackage{esdiff}
\usepackage{thmtools}
\usepackage{thm-restate}

\definecolor{azure(colorwheel)}{rgb}{0.0, 0.5, 1.0}
\definecolor{Pink}{RGB}{255, 105, 180}

\usetikzlibrary{arrows.meta}

\newtheorem{innercustomthm}{Theorem}
\newenvironment{customthm}[1]
  {\renewcommand\theinnercustomthm{#1}\innercustomthm}
  {\endinnercustomthm}

\newtheorem*{thm*}{Theorem}
\newtheorem{thm}{Theorem}
\newtheorem{lem}[thm]{Lemma}
\newtheorem{pro}[thm]{Proposition}

\newtheorem{ques}[thm]{Question}

\crefname{thm}{Theorem}{Theorems}
\crefname{lem}{Lemma}{Lemmas}

\newcommand{\N}{\mathbb{N}}
\newcommand{\Z}{\mathbb{Z}}

\newcommand{\col}{\mathrm{col}}

\newcommand\set[1]{\left\{ #1 \right\}}

\newcommand\ceil[1]{\left\lceil #1 \right\rceil}

\allowdisplaybreaks

\begin{document}

\title{Covering Families for DP-Coloring of Cartesian Products with Complete Bipartite Graphs}

\author{
Hemanshu Kaul \thanks{School of Computing, Illinois Institute of Technology, Chicago, IL, USA (kaul@illinoistech.edu)}
\and
Jeffrey A. Mudrock \thanks{Department of Mathematics and Statistics, University of South Alabama, Mobile, AL, USA (mudrock@southalabama.edu)}
\and
Emily A. Psyhogios \thanks{Lake Forest College, Lake Forest, IL, USA (psyhogiosear@lakeforest.edu)}
\and
Gunjan Sharma \thanks{Department of Mathematics and Computer Science, Lake Forest College, Lake Forest, IL, USA (gsharma@lakeforest.edu)}
\and
Illia Siutkin \thanks{Lake Forest College, Lake Forest, IL, USA (siutkini@lakeforest.edu)}
\and
Aparna Upadhyay \thanks{Department of Mathematics and Statistics, University of South Alabama, Mobile, AL, USA (aupadhyay@southalabama.edu)}
}

\maketitle

\begin{abstract}
A famous folklore result in list coloring demonstrating that the gap between the list chromatic number and chromatic number of a graph can be arbitrarily large is: $\chi_{\ell}(K_{l,t}) = 1+l$ if and only if $t \geq l^l$. DP-coloring (also called correspondence coloring) is a well-studied generalization of list coloring introduced in 2015. In 2018, Mudrock studied the DP analogue of the aforementioned folklore result. He proved that for $l \in\mathbb{N}$, if $\mu(l)$ is the smallest integer $t$ such that $\chi_{DP}(K_{l,t})=1+l$, then $\left\lceil l^l/l!\right\rceil \leq \mu(l) \leq 1+l^l(\log(l!)+1)/l!$. Recently, Kaul, Mudrock, and Sharma studied a more general version of this problem by studying the smallest $t$ for which $\chi_{DP}(G \square K_{l,t}) = k + l$, where $G$ satisfies certain criticality conditions and $G \square K_{l,t}$ denotes the Cartesian product of $G$ and $K_{l,t}$. 

In this paper, we introduce a notion we call \emph{covering families} that gives a new perspective on these DP-coloring questions.  In particular, if $\kappa(l)$ denotes the minimum size of a covering family of $[l]^l$, we show that $\mu(l)=\kappa(l)$.  We use this equivalence to prove $\mu(4)=12$ and to obtain new general lower bounds on $\mu(l)$. We also prove a general upper bound on the minimum size of covering families which yields an improved general upper bound on $\mu(l)$ and gives improvements on known bounds for related DP-coloring questions involving Cartesian products with complete bipartite graphs.

\medskip



\end{abstract}

\section{Introduction}\label{intro}

In this paper all graphs are nonempty, finite, and simple unless otherwise noted. Generally speaking, we follow West~\cite{W01} for terminology and notation. We use $\mathbb{N}$ to denote the set of all natural numbers. For $k \in \mathbb{N}$, $[k]$ denotes the set $\{1,\ldots,k\}$, $\Z_k$ denotes the integers mod $k$, and $[0]$ denotes the empty set.  For a positive real number $x$, $\log(x)$ is the natural log of $x$. For a graph $G$, $V(G)$ and $E(G)$ are the vertex set and edge set of $G$, respectively. If $S \subseteq V(G)$, $G[S]$ is the subgraph of $G$ induced by $S$. For any $S_{1},S_{2} \subseteq V(G)$, $E_{G}(S_{1},S_{2})$ denotes the set consisting of the edges in $E(G)$ that have one endpoint in $S_{1}$ and the other in $S_{2}$. The neighborhood of a vertex $v$ in $G$ is denoted by $N_{G}(v)$ or $N(v)$ when the graph is clear from context. The neighborhood of a set of vertices $S\subseteq V(G)$ is defined as $N(S) = \bigcup_{v \in S}N_{G}(v)$. The degree of a vertex $v\in V(G)$ is defined as the number of vertices in $G$ adjacent to $v$; we denote this by $d_G(v)$ or simply $d(v)$ when the graph is clear from context.

\subsection{Graph coloring, list coloring, and DP-coloring}

A \emph{proper $k$-coloring} of a graph $G$ is a function $f$ that assigns an element of $[k]$ to each $v \in V(G)$ such that $f(v) \neq f(u)$ whenever $uv \in E(G)$. We say that $G$ is \emph{$k$-colorable} if it has a proper $k$-coloring. The \emph{chromatic number} of $G$, denoted by $\chi(G)$, is the smallest $k \in \mathbb{N}$ such that there exists a proper $k$-coloring of $G$. The \emph{coloring number} of a graph $G$, denoted by $\col(G)$, is the smallest integer $d$ for which there exists an ordering, $v_1, \ldots, v_n$, of the vertices of $G$ such that each vertex $v_i$ has at most $d-1$ neighbors among $v_1, \ldots, v_{i-1}$. For example, $\col(K_{l,t}) = \min\{l,t\}+1$.

List coloring is a generalization of classical vertex coloring. It was introduced in the 1970s independently by Vizing~\cite{V76} and Erd\H{o}s, Rubin, and Taylor~\cite{ET79}. A \emph{list assignment} of $G$ is a function $L$ with domain $V(G)$ that assigns a set of colors to each $v \in V(G)$. If $|L(v)| = k$ for each $v \in V(G)$, then $L$ is called a \emph{$k$-assignment} of $G$. The graph $G$ is \emph{$L$-colorable} if there exists a proper coloring $f$ of $G$ such that $f(v) \in L(v)$ for each $v \in V(G)$ (we refer to $f$ as a \emph{proper $L$-coloring} of $G$). The \emph{list chromatic number} of $G$, denoted by $\chi_{\ell}(G)$, is the smallest $k$ such that there exists a proper $L$-coloring of $G$ for every $k$-assignment $L$ of $G$. It immediately follows that for any graph $G$, $\chi(G) \leq \chi_\ell(G) \leq \col(G)$. The first inequality may be strict since it is known that the gap between $\chi(G)$ and $\chi_{\ell}(G)$ can be arbitrarily large; for example, $\chi_{\ell}(K_{l,t}) = l+1$ if and only if $t \geq l^l$, but all bipartite graphs are $2$-colorable.
 
DP-coloring (also called correspondence coloring) is a generalization of list coloring that was introduced by Dvo\v{r}\'{a}k and Postle~\cite{DP15} in 2015. Intuitively, DP-coloring is a generalization of list coloring where each vertex in the graph still gets a list of colors, but the identification of which colors are different can vary from edge to edge. We now give formal definitions. A \emph{DP-cover} of a graph $G$ is a pair $\mathcal{H} = (L,H)$, where:

\begin{itemize}
\item $H$ is a graph and $L$ is a function assigning to each $v \in V(G)$ a set $L(v) \subseteq V(H)$,

\item the sets $L(v)$ for $v \in V(G)$ are disjoint sets, independent sets in $H$, and satisfy $V(H) = \bigcup_{v \in V(G)} L(v)$, 

\item if $u \neq v$ and $E_H(L(u),L(v)) \neq \emptyset$, then $uv \in E(G)$ and $E_H(L(u), L(v))$ is a matching.

\end{itemize}

We stress that the matchings between $L(u)$ and $L(v)$ for $uv \in E(G)$ need not be perfect (and may even be empty). The vertices of $H$ are referred to as \emph{colors}. In this paper, cover always refers to a DP-cover. Suppose $\mathcal{H} = (L,H)$ is a cover of $G$. An \emph{$\mathcal{H}$-coloring} of $G$ is an \emph{independent transversal} of $\mathcal{H}$ which is an independent set $I$ in $H$ of size $|V(G)|$ with the property $|I \cap L(u)|=1$ for each $u \in V(G)$. We say $\mathcal{H}$ is \emph{$k$-fold} if $|L(u)|=k$ for each $u \in V(G)$. We say $\mathcal{H}$ is a \emph{bad cover} of $G$ if $G$ does not have an $\mathcal{H}$-coloring.  Note that a bad cover of $G$ need not be $k$-fold for some $k \in \N$.  If $\mathcal{H}$ is a bad cover of $G$ and $\mathcal{H}$ is $k$-fold, we say that $\mathcal{H}$ is a \emph{bad $k$-fold cover} of $G$. A $k$-fold cover $\mathcal{H}$ is a \emph{full cover} if for each $uv \in E(G)$, the matching $E_{H}(L(u),L(v))$ is perfect. The \emph{DP-chromatic number} of $G$, denoted by $\chi_{DP}(G)$, is the smallest $k \in \N$ such that $G$ has an $\mathcal{H}$-coloring whenever $\mathcal{H}$ is a $k$-fold cover of $G$.

Given a $k$-assignment $L$ for a graph $G$, it is easy to construct a $k$-fold cover $\mathcal{H}$ of $G$ such that $G$ has an $\mathcal{H}$-coloring if and only if $G$ has a proper $L$-coloring. It follows that $\chi_\ell(G) \leq \chi_{DP}(G)$. In general, for any graph $G$, $\chi(G) \leq \chi_\ell(G) \leq \chi_{DP}(G) \leq \col(G)$, and these inequalities may be strict.

The example of complete bipartite graphs mentioned above naturally leads to the analogous question for DP-coloring. In 2018, Mudrock~\cite{M18} studied this question. For $l\in\N$, let $\mu(l)$ denote the smallest integer $t$ such that $\chi_{DP}(K_{l,t})=1+l.$ Mudrock proved the following.
\begin{thm}[\cite{M18}] \label{thm: mudrock}
For each $l\in\N$, $\left\lceil l^l/l!\right\rceil
\leq \mu(l)
\leq
1+l^l(\log(l!)+1)/l!.$
\end{thm}
It is also mentioned in~\cite{M18} that $\mu(1)=1$, $\mu(2)=2$, and $\mu(3)=6$. One of the motivations for the present work is to better understand the combinatorial structures underlying constructions of bad covers for these types of problems.

\subsection{Counting Colorings}
For $k \in \N$, the \emph{chromatic polynomial} of a graph $G$, denoted by $P(G,k)$, is equal to the number of proper $k$-colorings of $G$. As the name suggests, it can be shown that $P(G,k)$ is a polynomial in $k$ of degree $|V(G)|$ (see~\cite{B12}).  In 1990 the notion of chromatic polynomial was extended to list coloring as follows~\cite{AS90}. If $L$ is a list assignment for $G$, we use $P(G,L)$ to denote the number of proper $L$-colorings of $G$. The \emph{list color function} $P_\ell(G,k)$ is the minimum value of $P(G,L)$ where the minimum is taken over all possible $k$-assignments $L$ for $G$.  Since a $k$-assignment could assign the same $k$ colors to every vertex in a graph, it is clear that $P_\ell(G,k) \leq P(G,k)$ for each $k \in \N$. 

We now turn our attention to DP coloring.  Suppose $\mathcal{H} = (L,H)$ is a $k$-fold cover of $G$. Then, $\mathcal{H}$ is \emph{canonical} if it admits a \emph{canonical labeling}, which is a mapping $\lambda \colon V(H) \to [k]$ such that
\begin{itemize}
\item for each $v \in V(G)$, the restriction of ${\lambda}$ to $L(v)$ is a bijection from $L(v)$ to $[k]$, and
\item for all $uv \in E(G)$ and $c \in L(u)$, $c' \in L(v)$, we have $cc' \in E(H)$ if and only if $\lambda(c) = \lambda(c')$.
\end{itemize}
Suppose $k \in \N$, and $\mathcal{H} = (L,H)$ is a $k$-fold cover of $G$ such that $\mathcal{H}$ has a canonical labeling called $\lambda$. From this point forward, unless otherwise noted, in such a situation we will let $(v,j)$ be the vertex in $L(v)$ that $\lambda$ maps to $j$ for each $v \in V(G)$ and $j \in [k]$. Notice that if $\mathcal{I}$ is the set of $\mathcal{H}$-colorings of $G$ and $\mathcal{C}$ is the set of proper $k$-colorings of $G$, the function $f: \mathcal{C} \rightarrow \mathcal{I}$ given by $f(c) = \{(v,c(v)) : v \in V(G)\}$ is a bijection.

The notion of chromatic polynomial was extended to DP-coloring in~\cite{KM19}.  Suppose $\mathcal{H} = (L,H)$ is a cover of graph $G$.  Let $P_{DP}(G, \mathcal{H})$ be the number of $\mathcal{H}$-colorings of $G$.  Then, the \emph{DP color function of $G$}, denoted by $P_{DP}(G,k)$, is the minimum value of $P_{DP}(G, \mathcal{H})$ where the minimum is taken over all possible $k$-fold covers $\mathcal{H}$ of $G$.  It is easy to see that for any graph $G$ and $k \in \N$, $P_{DP}(G, k) \leq P_\ell(G,k) \leq P(G,k)$.

\subsection{Criticality and Cartesian Products} 

Criticality is a notion of fundamental importance in extremal graph theory that is used to study a wide variety of graph properties. For $k\geq 2$, a \emph{$k$-critical graph} is a graph whose chromatic number is $k$ but whose proper subgraphs have chromatic number strictly less than $k$.  For convenience, we regard $K_1$ as the unique $1$-critical graph. We will also refer to a $k$-critical graph as a \emph{critical graph}. In 1951, Dirac~\cite{D51} initiated the study of critical graphs and since then this notion has been widely investigated.

The \emph{Cartesian product} of graphs $G$ and $H$, denoted by $G \square H$, is the graph with vertex set $V(G) \times V(H)$ and edges created so that $(u,v)$ is adjacent to $(u',v')$ if and only if either $u=u'$ and $vv' \in E(H)$ or $v=v'$ and $uu' \in E(G)$. Also, it is well-known that $\chi(G \square H) = \max\{\chi(G), \chi(H)\}$.  The following result concerning the DP-chromatic number of Cartesian products was proved in~\cite{KMG21}.
\begin{thm} [\cite{KMG21}] \label{thm: cartprod}
For any graphs $G$ and $H$, $\chi_{DP}(G\square H) \leq \min\{\chi_{DP}(G) + \col(H),\chi_{DP}(H) + \col(G)\} - 1.$
\end{thm}
For any graph $G$, Theorem~\ref{thm: cartprod} implies $\chi_{DP}(G\square K_{l,t})\leq \chi_{DP}(G)+l.$ It is natural to ask how large $t$ must be for equality to hold. Towards this, we have the following result.
\begin{thm} [\cite{KMG21}] \label{thm: cartprodcompbipartite}
For any graph $G$, $\chi_{DP}(G\square K_{l,t}) = \chi_{DP}(G) + l$ whenever $t \geq (P_{DP}(G,\chi_{DP}(G)+l-1))^{l}.$
\end{thm}
The following question is now natural.
\begin{ques} [\cite{KMG21}] \label{ques: 3}
Given a graph $G$ and $l \in \N$, let $f_{DP}(G,l)$ be the function satisfying: $\chi_{DP}(G\square K_{l,t})=\chi_{DP}(G)+l$
if and only if $t \geq f_{DP}(G,l)$. What is $f_{DP}(G,l)$?
\end{ques}

Since $\chi_{DP}(G\square K_{l,0}) = \chi_{DP}(G) < \chi_{DP}(G) + l$, we have $f_{DP}(G,l)\geq 1$. Moreover, by Theorem~\ref{thm: cartprodcompbipartite}, $
f_{DP}(G,l) \leq (P_{DP}(G,\chi_{DP}(G)+l-1))^{l}.$  Also, since $G\square K_{l,t}$ is a subgraph of $G\square K_{l,t+1}$, $\chi_{DP}(G\square K_{l,t})$ is nondecreasing in $t$. Hence $f_{DP}(G,l)$ exists for every graph $G$ and $l \in \N$. Also, note that $f_{DP}(K_1,l)=\mu(l)$. 

In \cite{KMS25}, upper bounds on $f_{DP}(G,l)$ were studied for $k$-critical graphs $G$ satisfying $\chi_{DP}(G)=k$. In particular, if $G$ is a $k$-critical graph such that $\chi_{DP}(G)=k$, then $f_{DP}(G,1)\leq P(G,k)/k$. For $l\geq 2$, the following result was proved in \cite{KMS25} with a straightforward probabilistic argument.
\begin{thm} [\cite{KMS25}] \label{thm: old}
For $k \geq 2$, let $G$ be a $k$-critical graph such that $\chi_{DP}(G) = k$. For $l \geq 2$, let
\[
c_{k,l}
=
\left\lceil
\displaystyle\frac{l\log(k+l-1)}
{\log((k-1)!)+l\log(k+l-1)-\log((k-1)!(k+l-1)^{l}-(k+l-1)!)}
\right\rceil.
\]
Then, $\chi_{DP}(G \square K_{l,t}) = k+l$ whenever
\[
t \geq c_{k,l}
\left(
\displaystyle\frac{P(G,k+l-1)}{k+l-1}
\right)^l.
\]
\end{thm}
Theorem~\ref{thm: old} gives a sufficient condition for $\chi_{DP}(G \square K_{l,t})=k+l$, rather than an if and only if condition. Consequently, one of the focuses of this paper is on obtaining improvements to this sufficient condition.  In particular, we will obtain improvements on the bound in Theorem~\ref{thm: old} for various values of $k$ and $l$ with $l\geq 3$. We also obtain improved general upper and lower bounds on $\mu(l)$, improving the bounds in Theorem~\ref{thm: mudrock}. For each fixed $l\geq 3$, the quantity inside the ceiling in the definition of $c_{k,l}$ approaches $l$ from above as $k\to\infty$. Consequently, $c_{k,l}=l+1$ for all sufficiently large $k$.

\subsection{Outline of Paper and Results} 

We now present an outline of the paper.  In Section~\ref{cover} we introduce the notion of \emph{covering family}.  Specifically, suppose $l \in \N$ and $d$ is a nonnegative integer.  Let $S_{[l+d]}$ denote the symmetric group on $[l+d]$.  Suppose $C \subseteq (S_{[l+d]})^l$.  We say that $C$ is a \emph{covering family of $[l+d]^l$} if for each $(x_1,\dots,x_l) \in [l+d]^l$ there exists a $(\sigma_1, \dots, \sigma_l) \in C$ such that $\sigma_1(x_1), \dots, \sigma_l(x_l)$ are pairwise distinct.  We then use this notion to prove a result in the form of Theorem~\ref{thm: old}.  
\begin{thm} \label{thm: new}
For $k \in \N$, let $G$ be a graph such that $\chi(G) = \chi_{DP}(G) = k$. For $l \geq 2$, suppose $C$ is a covering family of $[l+k-1]^l$ and $c = |C|$. Then, $\chi_{DP}(G \square K_{l,t}) = k+l$ whenever $t \geq c \left(\frac{P(G,k+l-1)}{k+l-1}\right)^l$. 
\end{thm}
Notice that Theorem~\ref{thm: new} does not require $G$ to be critical. In particular, since every $k$-critical graph has chromatic number $k$, Theorem~\ref{thm: new} applies to every graph satisfying the hypotheses of Theorem~\ref{thm: old}.  We also notice that Theorem~\ref{thm: new} gives us a way to study the upper bound in Theorem~\ref{thm: mudrock} and, when $l\geq 3$, the bound in Theorem~\ref{thm: old} using covering families. When it comes to Theorem~\ref{thm: mudrock}, the situation is particularly nice.  For $l\in\N$, let $\kappa(l)$ denote the minimum size of a covering family of $[l]^l$. As a consequence of Theorem~\ref{thm: new}, together with a converse argument, we obtain the following.
\begin{thm} \label{thm: mucover}
For each $l\in\N$, $\mu(l)=\kappa(l)$.
\end{thm}
Thus, determining $\mu(l)$ is equivalent to determining the minimum size of a covering family of $[l]^l$.  This means that the minimum size of a covering family of $[l]^l$ is 1, 2, and 6 when $l$ is 1, 2, and 3 respectively. In Section~\ref{Gunjan}, we use Theorem~\ref{thm: mucover} to determine $\mu(4)$.  Note that Theorem~\ref{thm: mudrock} gives $11 \leq \mu(4) \leq 45$. We construct a covering family of $[4]^4$ of size 12 and then use integer programming to show that no covering family of $[4]^4$ of size at most 11 exists. Consequently,
$\mu(4)=12$ (see Proposition~\ref{pro: 4upper} below). We also use covering families to improve the lower bound in Theorem~\ref{thm: mudrock}.
\begin{thm} \label{thm: uptheoddlow}
For each $l \in \N$ with $l \geq 3$ and $s \in [l-1]$,
\[
\mu(l) = \kappa(l) \geq
\ceil{\frac{s! \, l^{l-s}}{l!} \ceil{\frac{l^s}{s!}}}.
\]
Consequently,
\[
\mu(l) = \kappa(l) \geq
\max_{s \in [l-1]}
\ceil{\frac{s! \, l^{l-s}}{l!} \ceil{\frac{l^s}{s!}}}.
\]
\end{thm}
Taking $s=1$ recovers the lower bound in Theorem~\ref{thm: mudrock} while other choices of $s$ give improvements for various values of $l$. For example, when $l\geq 5$ is odd, taking $s=2$ gives a clear improvement.

Finally, in Section~\ref{Aparna} we begin by using Theorem~\ref{thm: new} to improve the bound given by Theorem~\ref{thm: old} when $l = 3$ and $2 \leq k \leq 76$.
\begin{thm} \label{thm: Aparna}
For $k \in \N$ and $k \geq 2$, let $G$ be a graph such that $\chi(G) = \chi_{DP}(G) = k$. Then, $\chi_{DP}(G \square K_{3,t}) = k+3$ whenever $t \geq 4 \left(\frac{P(G,k+2)}{k+2}\right)^3$. 
\end{thm}

We end Section~\ref{Aparna} by proving a result that gives an upper bound on the size of the smallest covering families of $[l+k-1]^l$.  The result gives some improvements on $c_{k,l}$ in Theorem~\ref{thm: old} when $l\geq 3$ and the upper bound on $\mu(l)$ in Theorem~\ref{thm: mudrock}. 
\begin{thm}\label{thm: letzgoupper}
Suppose $l \geq 2$, and let $\lambda_{1,l}=-\log(1-l!/l^l)$. Then,
\[
\mu(l)\leq \left\lceil \frac{\log(l^l\lambda_{1,l}/2)}{\lambda_{1,l}} \right\rceil+\left\lfloor \frac{1}{\lambda_{1,l}}+\frac{l}{2}\right\rfloor.
\]
Furthermore, suppose $k \geq 2$ and $G$ is a graph such that $\chi(G) = \chi_{DP}(G)=k$. Let
\[
p_{k,l}=\frac{(k+l-1)!}{(k-1)!(k+l-1)^l}
\;\text{ and }\;
\lambda_{k,l}=-\log(1-p_{k,l}).
\]
Then, $\chi_{DP}(G \square K_{l,t})=k+l$ whenever
\[
t\geq \left(\left\lceil \frac{\log((k+l-1)^l\lambda_{k,l}/2)}{\lambda_{k,l}} \right\rceil+\left\lfloor \frac{1}{\lambda_{k,l}}+\frac{1}{2}\right\rfloor\right)\left(\frac{P(G,k+l-1)}{k+l-1}\right)^l.
\]
\end{thm}
The following tables illustrate the improvements given by the results in Section~\ref{Aparna}.  In the first table we compare the upper bound on $\mu(l)$ in Theorem~\ref{thm: mudrock} with the bound in Theorem~\ref{thm: letzgoupper}.  We also include the known exact values, the last of which follows from Proposition~\ref{pro: 4upper}.  The second table compares Theorems~\ref{thm: old}, \ref{thm: Aparna}, and \ref{thm: letzgoupper} when $l=3$.  The third table gives some additional comparisons between Theorems~\ref{thm: old} and \ref{thm: letzgoupper}.  There is no restriction relating $k$ and $l$ in Theorem~\ref{thm: letzgoupper}, so we include several examples with $k>l$ in the third table.  The improvements from Theorem~\ref{thm: letzgoupper} tend to be strongest when $k$ is small compared to $l$.  For each fixed $l\geq3$, the coefficient $c_{k,l}$ in Theorem~\ref{thm: old} and the corresponding coefficient in Theorem~\ref{thm: letzgoupper} are both equal to $l+1$ for all sufficiently large $k$.  In some intermediate ranges, Theorem~\ref{thm: old} can be slightly better.

\begin{center}
\begin{minipage}[t]{0.36\textwidth}
\centering
\scriptsize
\setlength{\tabcolsep}{2pt}
\begin{tabular}{c|c|c|c}
\hline
$l$ & Thm.~\ref{thm: mudrock} & Thm.~\ref{thm: letzgoupper} & Exact $\mu(l)$\\
\hline
$2$ & $4$ & $3$ & $2$\\
$3$ & $13$ & $10$ & $6$\\
$4$ & $45$ & $38$ & $12$\\
$5$ & $151$ & $134$ & ? \\
$6$ & $492$ & $446$ & ? \\
$7$ & $1557$ & $1443$ & ? \\
$8$ & $4829$ & $4539$ & ? \\
$9$ & $14736$ & $13993$ & ? \\
$10$ & $44380$ & $42467$ & ? \\
\hline
\end{tabular}
\end{minipage}
\hfill
\begin{minipage}[t]{0.39\textwidth}
\centering
\scriptsize
\setlength{\tabcolsep}{2pt}
\begin{tabular}{c|c|c|c}
\hline
$k$ & $c_{k,3}$ & Thm.~\ref{thm: Aparna} & Thm.~\ref{thm: letzgoupper}\\
\hline
$2$ & $9$ & $4$ & $8$\\
$3$ & $8$ & $4$ & $8$\\
$4\leq k\leq5$ & $7$ & $4$ & $7$\\
$6\leq k\leq9$ & $6$ & $4$ & $7$\\
$10\leq k\leq11$ & $6$ & $4$ & $6$\\
$12\leq k\leq19$ & $5$ & $4$ & $6$\\
$20\leq k\leq76$ & $5$ & $4$ & $5$\\
$77\leq k\leq156$ & $4$ & $4$ & $5$\\
$k\geq157$ & $4$ & $4$ & $4$\\
\hline
\end{tabular}
\end{minipage}
\hfill
\begin{minipage}[t]{0.23\textwidth}
\centering
\scriptsize
\setlength{\tabcolsep}{2pt}
\begin{tabular}{c|c|c}
\hline
$(k,l)$ & $c_{k,l}$ & Thm.~\ref{thm: letzgoupper}\\
\hline
$(2,4)$ & $31$ & $25$\\
$(6,4)$ & $15$ & $15$\\
$(15,4)$ & $10$ & $11$\\
$(2,5)$ & $93$ & $72$\\
$(8,5)$ & $26$ & $25$\\
$(50,5)$ & $12$ & $13$\\
$(2,6)$ & $267$ & $203$\\
$(10,6)$ & $43$ & $42$\\
$(50,6)$ & $18$ & $18$\\
\hline
\end{tabular}
\end{minipage}
\end{center}
\section{Covering Families} \label{cover}

In this section we first prove Theorem~\ref{thm: new} and then prove Theorem~\ref{thm: mucover}.  Suppose $l \in \N$ and $d$ is a nonnegative integer.  Let $S_{[l+d]}$ denote the symmetric group on $[l+d]$.  Suppose $C \subseteq (S_{[l+d]})^l$.  We say that $C$ is a \emph{covering family of $[l+d]^l$} if for each $(x_1,\dots,x_l) \in [l+d]^l$ there exists a $(\sigma_1, \dots, \sigma_l) \in C$ such that $\sigma_1(x_1), \dots, \sigma_l(x_l)$ are pairwise distinct.  When $\sigma_1(x_1), \dots, \sigma_l(x_l)$ are pairwise distinct, we say that $(\sigma_1, \dots, \sigma_l)$ \emph{covers} $(x_1, \dots, x_l)$.  Consequently, $C$ is a covering family of $[l+d]^l$ if and only if every element of $[l+d]^l$ is covered by some element of $C$.

We now need some notions from DP-coloring.  Suppose $\mathcal{H} = (L,H)$ is a cover of a graph $G$. Suppose $U \subseteq V(G)$. Let $\mathcal{H}_{U} = (L_{U},H_{U})$ where $L_{U}$ is the restriction of $L$ to $U$ and $H_{U} = H[\bigcup_{u \in U} L(u)]$. Clearly, $\mathcal{H}_{U}$ is a cover of $G[U]$. We call $\mathcal{H}_{U}$ the \emph{subcover of }$\mathcal{H}$\emph{ induced by }$U$. Suppose $G'$ is a subgraph of $G$. Let $\mathcal{H}_{G'} = (L_{G'},H_{G'})$ where $L_{G'}$ is the restriction of $L$ to $V(G')$ and $H_{G'}=H[\bigcup_{u \in V(G')} L(u)] - \bigcup_{uv \in E(G) - E(G')}E_H(L(u),L(v))$. We call $\mathcal{H}_{G'}$ the \emph{subcover of }$\mathcal{H}$\emph{ corresponding to }$G'$.

For the rest of this section we will assume $G$ is a graph with $\chi(G) = \chi_{DP}(G)=k$ for some $k\in\N$ and $V(G)=\{u_1,\ldots,u_n\}$. Suppose $r\geq k$, and let $\mathcal{C}$ denote the set of all proper $r$-colorings of $G$. We define an equivalence relation $R$ on $\mathcal{C}$ such that if $g,h\in\mathcal{C}$, then $gRh$ if there exists $j\in\mathbb{Z}_r$
such that $(g(u_i)-h(u_i))\mod r=j$ for all $i\in[n]$. The following lemma is now immediate.

\begin{lem} \label{lem: latequiclass}
Each equivalence class $\mathcal{E}$ of $R$ as defined above is of size
$r$. Furthermore, if $g,h\in\mathcal{E}$ satisfy $g\neq h$, then
$g(u_i)\neq h(u_i)$ for all $i\in[n]$.
\end{lem}

Note that Lemma~\ref{lem: latequiclass} gives a partition of the set of all proper $r$-colorings of $G$ into sets of size $r$ which immediately gives a corresponding partition of proper $\mathcal{H}$-colorings of $G$ into sets of size $r$ when $\mathcal{H}$ is a canonical $r$-fold cover of $G$. 

We wish to study the DP-chromatic number of the Cartesian product of $G$ and a complete bipartite graph.  So, let $K$ be a copy of the complete bipartite graph $K_{l,t}$ with partite sets $X = \{x_{j} : j \in [l]\}$ and $Y = \{y_{q} : q \in [t]\}$. Let $M = G \square K$, $M_{X} = M[\{(u_{i},x_{j}) : i \in [n], j \in [l]\}]$, and for each $q \in [t]$, let $M_{y_{q}} = M[\{(u_{i},y_{q}) : i \in [n]\}]$. 

Now, suppose $\mathcal{H} = (L,H)$ is a $(k+l-1)$-fold cover of $M$.  We recall the definition of \emph{volatile coloring} from~\cite{KMG21}. Intuitively, the notion of volatile coloring formalizes the natural obstruction to extending a partial coloring $I$ of $M_X$ to the rest of the graph $M$. As we will see in Lemma~\ref{lem: nohcoloring}, this turns out to be the only obstruction to coloring the graph $M$.

Let $\mathcal{H}_{X} = (L_{X}, H_{X})$ denote the subcover of $\mathcal{H}$ induced by $V(G)\times X$. Similarly, for each $q \in [t]$, let $\mathcal{H}_{y_{q}} = (L_{y_{q}}, H_{y_{q}})$ denote the subcover of $\mathcal{H}$ induced by $V(G)\times \{y_{q}\}$.  Suppose $I$ is an $\mathcal{H}_{X}$-coloring of $M_{X}$. For each $q \in [t]$, let $D_{q} = \{w \in V(H_{y_{q}}) : N_{H}(w)\cap I = \emptyset\}$ and $H'_{y_{q}} = H[D_{q}]$ (note it is possible that $D_q = \emptyset$). For each $v \in V(M_{y_{q}})$, let $L'_{y_{q}}(v) = L_{y_{q}}(v)\cap D_{q}$ and $\mathcal{H}'_{y_{q}} = (L'_{y_{q}}, H'_{y_{q}})$.  We say $I$ is \emph{volatile} for $M_{y_{q}}$, if $\mathcal{H}'_{y_{q}}$ is a bad cover of $M_{y_{q}}$.  The following result from~\cite{KMG21} will be useful for us.

\begin{lem} [\cite{KMG21}] \label{lem: nohcoloring}
The cover $\mathcal{H}$ is a bad cover of $M$ if and only if for every $\mathcal H_X$-coloring $I_X$ of $M_X$, there exists $q\in[t]$ such that $I_X$ is volatile for $M_{y_q}$.
\end{lem}

We are now ready to prove Theorem~\ref{thm: new} which we restate.

\begin{customthm}{\ref{thm: new}}
For $k \in \N$, let $G$ be a graph such that $\chi(G) = \chi_{DP}(G) = k$. For $l \geq 2$, suppose $C$ is a covering family of $[l+k-1]^l$ and $c = |C|$. Then, $\chi_{DP}(G \square K_{l,t}) = k+l$ whenever
\[
t \geq c \left(\frac{P(G,k+l-1)}{k+l-1}\right)^l.
\]
\end{customthm}

\begin{proof}
Let the vertices of $G$ be ordered as $u_{1},\hdots,u_{n}$, and let $K$ be the complete bipartite graph with bipartition $X = \{x_{j} : j \in [l]\}$ and $Y = \{y_{q} : q \in [t]\}$. Let $M = G\square K$ and $M_{X} = M[\{(u_{i},x_{j}) : i \in [n], j \in [l]\}]$.  By Theorem~\ref{thm: cartprod}, we have $\chi_{DP}(M) \leq \chi_{DP}(G) + \col(K) - 1 \leq k+l$. It remains to show that $\chi_{DP}(M) > k+l-1$. We do this by constructing a bad $(k+l-1)$-fold cover $\mathcal{H} = (L,H)$ of $M$.    

For each $v \in V(K)$, let $\mathcal{H}_{v} = (L_{v}, H_{v})$ be a canonical $(k+l-1)$-fold cover of $M[V(G)\times\{v\}]$.  For each $v \in V(K)$, we let $L(u,v) = L_{v}(u,v)$ for every $u \in V(G)$ and create edges so that $H[\bigcup_{u \in V(G)} L(u,v)] = H_{v}$. For simplicity, for each $\lambda \in [k+l-1]$, we denote each vertex $((u,v),\lambda) \in V(H)$ as $(u,v,\lambda)$. 

Now, we will use $C$ to construct matchings (possibly empty) between $L(u_i,x_j)$ and $L(u_i,y_q)$ for each $i\in[n]$, $j\in[l]$, and $q\in[t]$ which will complete the construction of $\mathcal H$; any matchings not specified below are taken to be empty.  We suppose that $C = \{\pi_1, \dots, \pi_c \}$ where $\pi_i = (\sigma_{i,1},\dots,\sigma_{i,l})$ for each $i \in [c]$.

Let $\mathcal{H}_{X} = (L_{X},H_{X})$ denote the cover of $M_{X}$ where $L_{X}(u,x_{j}) = L_{x_{j}}(u,x_{j})$ for every $(u,x_{j}) \in V(G) \times X$ and $H_{X} = \bigcup_{j=1}^{l} H_{x_{j}}$. Note that $H_{x_{1}}, H_{x_{2}}, \ldots, H_{x_{l}}$ are pairwise vertex disjoint. Let $\mathcal{C}$ and $\mathcal{I}$ denote the collection of all proper $(k+l-1)$-colorings of $M_X$ and the collection of all $\mathcal{H}_{X}$-colorings of $M_{X}$ respectively. Note that since $\mathcal{H}_{X}$ is a canonical $(k+l-1)$-fold cover of $M_{X}$, the function $f: \mathcal{C} \rightarrow \mathcal{I}$ where $f(h) = \{(u_{i},x_{j},h(u_{i},x_{j})) : i\in [n], j \in [l]\}$ is a bijection. Also $P_{DP}(M_{X},\mathcal{H}_{X}) = P(M_{X},k+l-1)$. Let $d = P(G,k+l-1)$ so that $P_{DP}(M_{X},\mathcal{H}_{X}) = d^l$.

For each $j \in [l]$, let $\mathcal{C}_{j}$ and $\mathcal{I}_{j}$ denote the collection of proper $(k+l-1)$-colorings of $M[V(G)\times \{x_{j}\}]$ and the collection of $\mathcal{H}_{x_{j}}$-colorings of $M[V(G)\times \{x_{j}\}]$ respectively. Note that for each $j \in [l]$, the function $f_{j}: \mathcal{C}_{j} \rightarrow \mathcal{I}_{j}$ where $f_{j}(h) = \{(u_{i},x_{j},h(u_{i},x_{j})) : i\in [n]\}$ is a bijection. For each $j \in [l]$, $\mathcal{C}_{j}$ is partitioned into equivalence classes by the equivalence relation $R$ defined above. Since $d=P(G,k+l-1)$, Lemma~\ref{lem: latequiclass} implies that
$k+l-1$ divides $d$. We let $b = d/(k+l-1)$ and arbitrarily name these equivalence classes $\mathcal{E}_{j,p}$ where $p \in [b]$. So, $\{f_{j}(\mathcal{E}_{j,p}) : p \in [b] \}$ is a partition of $\mathcal{I}_j$.

Now, arbitrarily name the elements of the set $[b]^l$ as $\boldsymbol{p}_{1}, \ldots, \boldsymbol{p}_{b^l}$.  For each $m \in [b^l]$, let $C_{\boldsymbol{p}_m}$ consist of each $h \in \mathcal{C}$ such that for each $j \in [l]$ the restriction of $h$ to $V(G) \times \{x_j\}$ is an element of $\mathcal{E}_{j,z}$ where $z$ is the $j^{th}$ coordinate of $\boldsymbol{p}_m$.  Note that each element of $C_{\boldsymbol{p}_m}$ can be constructed as follows: for each $j \in [l]$ select a coloring in $\mathcal{E}_{j,z}$
where $z$ is the $j^{th}$ coordinate of $\boldsymbol{p}_m$,
then take the union of the $l$ selected colorings.  Since $|\mathcal{E}_{j,p}| = k+l-1$ for each $j \in [l]$ and $p \in [b]$, $|C_{\boldsymbol{p}_m}| = (k+l-1)^l$.  Moreover, $\left \{C_{\boldsymbol{p}_{m}} : m \in \left[b^l\right] \right \}$ is a partition of $\mathcal{C}$. Suppose $S_{\boldsymbol{p}_{m}} = f(C_{\boldsymbol{p}_m})$. Clearly $\left \{S_{\boldsymbol{p}_{m}} : m \in [b^l] \right \}$ is a partition of $\mathcal{I}$.

Note that by Lemma~\ref{lem: nohcoloring}, if $\mathcal{H}$ is such that for each $a \in \left[b^l\right]$, every $s \in S_{\boldsymbol{p}_{a}}$ is volatile for at least one of $M[V(G) \times \{y_{1}\}],\ldots, M[V(G) \times \{y_{t}\}]$, then $M$ does not have an $\mathcal{H}$-coloring. For each $a \in [b^l]$, we associate to $S_{\boldsymbol{p}_{a}}$ the following $c$ copies of $G$: $M[V(G) \times \{y_{c(a-1)+1}\}], \ldots, M[V(G) \times \{y_{c(a-1)+c}\}]$. Note that this can be done since $t \geq cb^l$.

Now, for each $a \in [b^l]$, we use $C$ to construct matchings between $L(u_{i},x_{j})$ and $L(u_{i},y_{q})$ for each $i \in [n], j \in [l]$, and $q \in \{c(a-1)+1,\ldots, c(a-1)+c\}$ so that each $s \in S_{\boldsymbol{p}_{a}}$ is volatile for at least one of $M[V(G) \times \{y_{c(a-1)+1}\}],\ldots, M[V(G) \times \{y_{c(a-1)+c}\}]$.

For each $a \in [b^{l}]$ we construct matchings as follows. Suppose that $\boldsymbol{p}_{a} = (p_{1}, \ldots, p_{l})$.  For each $j \in [l]$, $f_{j}(\mathcal{E}_{j,p_{j}})$ is a partition of the vertex set of $H_{x_{j}}$; by Lemma~\ref{lem: latequiclass}, we may suppose $f_{j}(\mathcal{E}_{j,p_{j}}) = \{I^{j,p_{j}}_{1},\ldots,I^{j,p_{j}}_{k+l-1}\}$. For each $j \in [l]$, $\omega \in [c]$, $i \in [n]$, and $z \in [k+l-1]$, we add an edge between the vertex in $I^{j,p_j}_{z} \cap L(u_{i},x_{j})$ and the vertex $(u_{i},y_{c(a-1)+\omega},\sigma_{\omega,j}(z))$.  This completes the construction of $\mathcal{H}$.

To show that $\mathcal{H}$ is indeed a bad $(k+l-1)$-fold cover of $M$, we will show that for arbitrary $a \in \left[b^l\right]$ and $s \in S_{\boldsymbol{p}_{a}}$, $s$ is volatile for at least one of: $M[V(G) \times \{y_{c(a-1)+1}\}], \ldots, M[V(G) \times \{y_{c(a-1)+c}\}]$.  Suppose $\boldsymbol{p}_{a} = (p_{1}, \ldots, p_{l})$ and for each $j \in [l]$, $f_{j}(\mathcal{E}_{j,p_{j}}) = \{I^{j,p_{j}}_{1},\ldots,I^{j,p_{j}}_{k+l-1}\}$.  Since each element of $S_{\boldsymbol{p}_{a}}$ can be constructed by choosing an element of $f_{j}(\mathcal{E}_{j,p_{j}})$ for each $j \in [l]$ and then taking the union of the selected elements, there is a $(a_1, \ldots, a_l) \in [k+l-1]^l$ such that 
$$s = \bigcup_{j=1}^l I^{j,p_{j}}_{a_j}.$$
Since $C$ is a covering family of $[l+k-1]^l$, there is a $\pi_o \in C$ such that $\sigma_{o,1}(a_1), \ldots, \sigma_{o,l}(a_l)$ are pairwise distinct.  We claim that $s$ is volatile for $M[V(G) \times \{y_{c(a-1)+o}\}]$.  To see why, first for each $v \in V(G)$, let
$$L'(v,y_{c(a-1)+o}) = L(v,y_{c(a-1)+o})-N_H(s) \text{ and } H'=H \left[\bigcup_{v \in V(G)} L'(v,y_{c(a-1)+o})\right].$$ 
Since $\sigma_{o,1}(a_1), \ldots, \sigma_{o,l}(a_l)$ are pairwise distinct, the construction of $\mathcal{H}$ implies that there is a $\Gamma \subseteq [k+l-1]$ with $|\Gamma|=k-1$ satisfying
$$L'(v,y_{c(a-1)+o}) = \{(v,y_{c(a-1)+o},\gamma) : \gamma \in \Gamma\}$$
for each $v \in V(G)$. When $k=1$, we have $L'(v,y_{c(a-1)+o})=\emptyset$ for each $v\in V(G)$, so $(L',H')$ is clearly a bad cover. When $k\geq 2$, $(L',H')$ is a canonical $(k-1)$-fold cover, so the fact that $\chi(G)=k$ implies that it is a bad cover. Thus, $s$ is volatile for $M[V(G) \times \{y_{c(a-1)+o}\}]$.
\end{proof}

We note that when $l=2$, there is a covering family of $[k+1]^2$ of size $2$ for every $k\in\N$. Indeed, let $\rho$ be a permutation of $[k+1]$ with no fixed points. Then, $\{(\mathrm{id},\mathrm{id}),(\mathrm{id},\rho)\}$ is a covering family of $[k+1]^2$.  If $a\neq b$, then $(\mathrm{id},\mathrm{id})$ covers $(a,b)$, and if $a=b$, then $(\mathrm{id},\rho)$ covers $(a,b)$. Consequently, Theorem~\ref{thm: new} implies that if $\chi(G)=\chi_{DP}(G)=k$, then $\chi_{DP}(G\square K_{2,t})=k+2$ whenever $t\geq 2(P(G,k+1)/(k+1))^2$.  To finish the section we prove Theorem~\ref{thm: mucover}, which we restate.

\begin{customthm}{\ref{thm: mucover}}
For each $l\in\N$, $\mu(l)=\kappa(l)$.
\end{customthm}

\begin{proof}
Clearly, the result holds when $l=1$. So, suppose $l\geq 2$. First, we show that $\mu(l)\leq \kappa(l)$. Let $C$ be a covering family of $[l]^l$ with $|C|=\kappa(l)$. Applying Theorem~\ref{thm: new} with $G=K_1$ and $k=1$, we have $\chi_{DP}(K_{l,t})=l+1$ whenever $t\geq \kappa(l)$ since $P(K_1,l)=l$. Thus, $\mu(l)\leq \kappa(l)$.

We now show that $\kappa(l)\leq \mu(l)$. Let $t=\mu(l)$, and let $\mathcal{H}=(L,H)$ be a bad $l$-fold cover of $K=K_{l,t}$. Let $X=\{x_1,\ldots,x_l\}$ and $Y=\{y_1,\ldots,y_t\}$ be the partite sets of $K$. By adding edges to $H$ if necessary, we may suppose that $E_H(L(x_i),L(y_j))$ is a perfect matching for each $i\in[l]$ and $j\in[t]$ since adding edges cannot create an $\mathcal{H}$-coloring.

For each $v\in V(K)$, arbitrarily label the elements of $L(v)$ with the elements of $[l]$. For each $i\in[l]$ and $j\in[t]$, let $\sigma_{j,i}$ be the permutation of $[l]$ such that the element of $L(x_i)$ labeled $a$ is matched to the element of $L(y_j)$ labeled $\sigma_{j,i}(a)$ for each $a\in[l]$. Finally, let $C=\{(\sigma_{j,1},\ldots,\sigma_{j,l}):j\in[t]\}$.  We claim that $C$ is a covering family of $[l]^l$ which will immediately imply $\kappa(l) \leq t = \mu(l)$.

Suppose $(a_1,\ldots,a_l)\in[l]^l$. Since $X$ is an independent set in $K$, the $l$-element set $I$ formed by selecting the element of $L(x_i)$ labeled $a_i$ for each $i\in[l]$ is an $\mathcal{H}_X$-coloring of $K[X]$. Since $\mathcal{H}$ is a bad cover, $I$ cannot be extended to an $\mathcal{H}$-coloring of $K$. 

Consequently, there must be some $j\in[t]$ such that every element in $L(y_j)$ has a neighbor in $I$. By definition of the permutations, this means: $\sigma_{j,1}(a_1),\ldots,\sigma_{j,l}(a_l)$ are pairwise distinct. Thus, $C$ is a covering family of $[l]^l$ and $\kappa(l)\leq\mu(l)$. The desired result immediately follows.
\end{proof}

Thus, the problem of determining $\mu(l)$ is precisely the problem of determining the minimum size of a covering family of $[l]^l$.

\section{Bounds on $\mu(l)$} \label{Gunjan}

In this section we begin by using Theorem~\ref{thm: mucover} to determine $\mu(4)$. Recall that Theorem~\ref{thm: mudrock} gives $11\leq \mu(4)\leq 45$. By Theorem~\ref{thm: mucover}, to determine $\mu(4)$ it suffices to first construct a covering family of $[4]^4$ of size 12, and then prove a covering family of $[4]^4$ of size 11 does not exist.

\begin{pro} \label{pro: 4upper}
We have $\mu(4) = 12$.
\end{pro}

\begin{proof}
First, we show that $\mu(4)\leq 12$. By Theorem~\ref{thm: mucover}, it suffices to construct a covering family of $[4]^4$ of size 12. For each $j\in[12]$, let $\pi_j=(\sigma_{j,1},\sigma_{j,2},\sigma_{j,3},\sigma_{j,4})\in(S_{[4]})^4.$  Using one-line notation for permutations of $[4]$, each of the twelve ordered quadruples is given in the table below.

\begin{center}
\begin{tabular}{c|c|c|c|c}
\hline
$j$ & $\sigma_{j,1}$ & $\sigma_{j,2}$ & $\sigma_{j,3}$ & $\sigma_{j,4}$ \\
\hline
$1$  & $(1,2,3,4)$ & $(3,2,4,1)$ & $(3,2,1,4)$ & $(3,1,4,2)$ \\
$2$  & $(1,2,3,4)$ & $(1,3,2,4)$ & $(3,2,1,4)$ & $(1,4,2,3)$ \\
$3$  & $(1,2,3,4)$ & $(1,3,2,4)$ & $(1,4,3,2)$ & $(2,4,1,3)$ \\
$4$  & $(1,2,3,4)$ & $(4,3,2,1)$ & $(2,3,4,1)$ & $(1,2,3,4)$ \\
$5$  & $(1,2,3,4)$ & $(2,1,3,4)$ & $(3,2,1,4)$ & $(2,1,3,4)$ \\
$6$  & $(1,2,3,4)$ & $(4,2,1,3)$ & $(3,2,1,4)$ & $(4,2,1,3)$ \\
$7$  & $(1,2,3,4)$ & $(4,2,1,3)$ & $(1,4,3,2)$ & $(1,2,4,3)$ \\
$8$  & $(1,2,3,4)$ & $(2,1,3,4)$ & $(1,4,3,2)$ & $(3,1,2,4)$ \\
$9$  & $(1,2,3,4)$ & $(2,1,4,3)$ & $(2,1,4,3)$ & $(1,4,3,2)$ \\
$10$ & $(1,2,3,4)$ & $(4,1,2,3)$ & $(4,3,2,1)$ & $(3,4,1,2)$ \\
$11$ & $(1,2,3,4)$ & $(2,3,4,1)$ & $(4,1,2,3)$ & $(3,2,1,4)$ \\
$12$ & $(1,2,3,4)$ & $(3,2,4,1)$ & $(1,4,3,2)$ & $(4,1,3,2)$ \\
\hline
\end{tabular}
\end{center}

Let $C=\{\pi_1,\ldots,\pi_{12}\}$.  For each $j\in[12]$, let $A_j$ consist of all $(a_1,a_2,a_3,a_4)\in[4]^4$ such that $\{\sigma_{j,1}(a_1),\sigma_{j,2}(a_2),
\sigma_{j,3}(a_3),\sigma_{j,4}(a_4)\}=[4].$ Clearly, $|A_j|=4!=24$ for each $j\in[12]$. A direct computation
shows that exactly 240 elements of $[4]^4$ belong to exactly one of the sets $A_j$, while the remaining 16 elements belong to exactly three of the sets $A_j$ (see Appendix~\ref{Appendmu4}). In particular, $\left|\bigcup_{j=1}^{12}A_j\right|=240+16=256=|[4]^4|$. Thus,
$\bigcup_{j=1}^{12}A_j=[4]^4$ which means $C$ is a covering family of $[4]^4$.

We now turn our attention to proving that $\mu(4)>11$. Our strategy for proving this is to set up an IP (i.e., an integer program). We begin with an important observation. Suppose that $\mathcal{C}=\{\pi_j:j\in[k]\}$ is a covering family of $[4]^4$, where $\pi_j=(\sigma_{j,1},\sigma_{j,2},\sigma_{j,3},\sigma_{j,4})\in(S_{[4]})^4$ for each $j\in[k]$. As above, for each $j\in[k]$, let $A_j$ consist of all $(a_1,a_2,a_3,a_4)\in[4]^4$ such that $\{\sigma_{j,1}(a_1),\sigma_{j,2}(a_2),\sigma_{j,3}(a_3),\sigma_{j,4}(a_4)\}=[4]$.

Now, for any $\psi\in S_{[4]}$, notice that $\{\sigma_{j,1}(a_1),\sigma_{j,2}(a_2),\sigma_{j,3}(a_3),\sigma_{j,4}(a_4)\}=[4]$ if and only if
$\{(\psi\circ\sigma_{j,1})(a_1),(\psi\circ\sigma_{j,2})(a_2),(\psi\circ\sigma_{j,3})(a_3),(\psi\circ\sigma_{j,4})(a_4)\}=[4]$. Thus, taking $\psi=\sigma_{j,1}^{-1}$, we see that every $A_j$ can be taken to correspond to an ordered quadruple whose first coordinate is the identity permutation. Consequently, if there is a covering family of $[4]^4$ of size $k$, then there is one of size at most $k$ with the property that the first coordinate of each element is the identity permutation.

We now use this fact to set up an IP. Let $\mathcal{S}$ be the set of ordered quadruples in $(S_{[4]})^4$ whose first coordinate is the identity permutation. Note that $|\mathcal{S}|=(4!)^3=13,824$. Index the elements of $\mathcal{S}$ so that $\mathcal{S}=\{\tau_1,\ldots,\tau_{13824}\}$. For each $j\in[13824]$, if $\tau_j=(\sigma_{j,1},\sigma_{j,2},\sigma_{j,3},\sigma_{j,4})$, let $B_j$ consist of all $(a_1,a_2,a_3,a_4)\in[4]^4$ such that $\{\sigma_{j,1}(a_1),\sigma_{j,2}(a_2),\sigma_{j,3}(a_3),\sigma_{j,4}(a_4)\}=[4]$.

For each $j\in[13824]$, let $x_j$ be a binary variable, where $x_j=1$ means that $B_j$ is selected. By what we have shown above, there is a covering family of $[4]^4$ of size at most 11 if and only if the following IP is feasible:
\[
x_j\in\{0,1\}\text{ for each }j\in[13824],\qquad
\sum_{\substack{j\in[13824]\\a\in B_j}}x_j\geq1\text{ for each }a\in[4]^4,\qquad
\sum_{j=1}^{13824}x_j\leq11.
\]
The second condition ensures that every element of $[4]^4$ is contained in at least one selected $B_j$, while the third condition ensures that at most 11 sets are selected. We solved this IP using HiGHS 1.15.1, which determined that the model is infeasible after 5934.81 seconds, 3720 nodes, and 7,107,177 LP iterations. Thus, there is no covering family of $[4]^4$ of size at most 11, and hence $\kappa(4)>11$. The code used to generate and solve this IP, together with the solver output, is available in~\cite{M26comp}.

Since the covering family constructed above shows that $\kappa(4)\leq 12$, we have $\kappa(4)=12$. The result now follows from Theorem~\ref{thm: mucover}.
\end{proof}

We now turn our attention to improving the lower bound in Theorem~\ref{thm: mudrock}.  We begin by presenting some notation that we will use for the remainder of the section.  For
$l \in \N$, $l \geq 3$, and $\pi = (\sigma_1, \ldots, \sigma_l) \in (S_{[l]})^l$, let
\[
A_{\pi}= \left\{(a_1, \ldots, a_l) \in [l]^l :
\sigma_1(a_1), \ldots, \sigma_l(a_l) \text{ are pairwise distinct}\right\}.
\]
Clearly, $C \subseteq (S_{[l]})^l$ is a covering family of $[l]^l$ if and only if $\bigcup_{\tau \in C} A_{\tau} = [l]^l$.  Since the function that maps each $(a_1, \ldots, a_l) \in [l]^l$ to $(\sigma_1(a_1), \ldots, \sigma_l(a_l))$ is a bijection, $|A_{\pi}|=l!$.

Now, suppose $s \in [l-1]$. For each $\boldsymbol{c} = (c_{s+1}, \ldots, c_l) \in
[l]^{l-s}$, let
\[
F_{\boldsymbol{c}} =
\set{(a_1, \ldots, a_s, c_{s+1}, \ldots, c_l) : a_1, \ldots, a_s \in [l]}.
\]
Clearly, $\{F_{\boldsymbol{c}} : \boldsymbol{c} \in [l]^{l-s} \}$ is a partition of $[l]^l$ into parts of size $l^s$.  For $\pi = (\sigma_1, \ldots, \sigma_l) \in (S_{[l]})^l$ it is immediately clear that $|A_{\pi} \cap F_{\boldsymbol{c}}| = s!$ when $\sigma_{s+1}(c_{s+1}), \ldots, \sigma_l(c_l)$ are pairwise distinct, and $|A_{\pi} \cap F_{\boldsymbol{c}}| = 0$ otherwise.  We are now ready to present a lemma and then prove Theorem~\ref{thm: uptheoddlow}.

\begin{lem} \label{lem: count1}
Suppose $l \in \N$, $l \geq 3$, $s \in [l-1]$, and $\pi \in (S_{[l]})^l$.  Then,
\[
\left| \left\{\boldsymbol{c} \in [l]^{\,l-s} :
A_{\pi} \cap F_{\boldsymbol{c}} \neq \emptyset\right\}\right| = \frac{l!}{s!}.
\]
\end{lem}

\begin{proof}
Suppose $\pi = (\sigma_1, \ldots, \sigma_l)$.  We know for each $\boldsymbol{c} = (c_{s+1}, \ldots, c_l) \in
[l]^{l-s}$, $A_{\pi} \cap F_{\boldsymbol{c}} \neq \emptyset$ if and only if $\sigma_{s+1}(c_{s+1}),
\ldots, \sigma_l(c_l)$ are pairwise distinct.  Also the function that maps each $\boldsymbol{c} = (c_{s+1}, \ldots, c_l) \in [l]^{l-s}$ to $(\sigma_{s+1}(c_{s+1}), \ldots, \sigma_l(c_l))$ is a bijection.  Consequently, the number of $\boldsymbol{c} \in [l]^{l-s}$ with the property $A_{\pi} \cap F_{\boldsymbol{c}} \neq \emptyset$ is the same as the number of strings with pairwise distinct entries of length $l-s$ with symbols taken from $[l]$.  The result immediately follows.
\end{proof}

\begin{customthm}{\ref{thm: uptheoddlow}}
For each $l \in \N$ with $l \geq 3$ and $s \in [l-1]$,
\[
\mu(l) = \kappa(l) \geq
\ceil{\frac{s! \, l^{l-s}}{l!} \ceil{\frac{l^s}{s!}}}.
\]
Consequently,
\[
\mu(l) = \kappa(l) \geq
\max_{s \in [l-1]}
\ceil{\frac{s! \, l^{l-s}}{l!} \ceil{\frac{l^s}{s!}}}.
\]
\end{customthm}

\begin{proof}
 Fix $s \in [l-1]$.  Suppose $C$ is a covering family of $[l]^l$.  We must show that $|C|$ is at least the desired lower bound.  Let
\[
\mathcal{D} = \left\{ (\pi, \boldsymbol{c}) \in C \times [l]^{\,l-s} :
A_{\pi} \cap F_{\boldsymbol{c}} \neq \emptyset \right\}.
\]
On the one hand, Lemma~\ref{lem: count1} implies that $|\mathcal{D}| = |C|l!/s!$.  On the other hand, since we know that for fixed $\boldsymbol{d} \in [l]^{l-s}$, $|A_{\pi} \cap F_{\boldsymbol{d}}| \leq s!$ for each $\pi \in C$, there are at least $\lceil l^s/s! \rceil$ elements $\pi \in C$ with the property $A_{\pi} \cap F_{\boldsymbol{d}} \neq \emptyset$.  Consequently,
\[
|\mathcal{D}| \geq l^{l-s} \ceil{\frac{l^s}{s!}}.
\]
Thus, we have that
\[
\frac{|C|l!}{s!} \geq l^{l-s} \ceil{\frac{l^s}{s!}}.
\]
Since $|C|$ is an integer the first bound immediately follows. Since this holds for each $s \in [l-1]$, the second bound follows as well. 
\end{proof}
We remark that when $l \geq 5$ is odd, taking $s=2$ gives
\[
\ceil{\frac{2l^{l-2}}{l!}\ceil{\frac{l^2}{2}}}
=
\ceil{\frac{l^l+l^{l-2}}{l!}},
\]
which is a clear improvement on the lower bound in Theorem~\ref{thm: mudrock}.  The following table compares the lower bounds we get from Theorems~\ref{thm: mudrock} and~\ref{thm: uptheoddlow}.  Note that Theorem~\ref{thm: uptheoddlow} provides no improvement when $l \in \{3,4,6\}$. 
\begin{center}
\begin{tabular}{c|c|c}
\hline
$l$ & Theorem~\ref{thm: mudrock} & Theorem~\ref{thm: uptheoddlow} \\
\hline
$5$  & $27$   & $28$ \\
$7$  & $164$  & $167$ \\
$8$  & $417$  & $420$ \\
$9$  & $1068$ & $1081$ \\
$10$ & $2756$ & $2762$ \\
$11$ & $7148$ & $7207$ \\
\hline
\end{tabular}
\end{center}

\section{Upper Bounds} \label{Aparna}

In this section we first prove Theorem~\ref{thm: Aparna}.  We begin by constructing a covering family of $[q]^3$ of size 4 for each $q\geq 4$.

\begin{lem} \label{lem: 3cover}
For each integer $q\geq 4$, there exists a covering family of $[q]^3$ of size 4.
\end{lem}

\begin{proof}
We first suppose that $q\neq 6$.  Then, $q$ has a prime-power divisor $r\geq 4$.  So, we may write $q=rm$, where $r\geq 4$ is a prime power.  We identify $[q]$ with $\mathbb{F}_r\times[m]$ where $\mathbb{F}_r$ denotes the finite field of order $r$.  Choose pairwise distinct elements $\alpha_1,\alpha_2,\alpha_3\in\mathbb{F}_r$, and let $T$ be an arbitrary subset of $\mathbb{F}_r$ of size 4.  For each $t\in T$ and $i\in[3]$, define a permutation $\sigma_{t,i}$ of $\mathbb{F}_r\times[m]$ by
\[
\sigma_{t,i}(x,j)=(x+\alpha_i t,j).
\]
For each $t\in T$, let $\pi_t=(\sigma_{t,1},\sigma_{t,2},\sigma_{t,3})$, and let $C=\{\pi_t:t\in T\}$.

We claim that $C$ is a covering family of $[q]^3$.  Suppose
\[
\boldsymbol{y}=((x_1,j_1),(x_2,j_2),(x_3,j_3))\in(\mathbb{F}_r\times[m])^3.
\]
We say that $t\in T$ is bad for $\boldsymbol{y}$ if
\[
\sigma_{t,1}(x_1,j_1),\sigma_{t,2}(x_2,j_2),
\sigma_{t,3}(x_3,j_3)
\]
are not pairwise distinct.  If $t$ is bad for $\boldsymbol{y}$, there are distinct $a,b\in[3]$ such that
\[
x_a+\alpha_a t=x_b+\alpha_b t.
\]
Equivalently,
\[
(\alpha_a-\alpha_b)t=x_b-x_a.
\]
Since $\alpha_a\neq\alpha_b$, this equation has at most one solution $t\in\mathbb{F}_r$.  There are three pairs of distinct elements of $[3]$; so, there are at most three values of $t\in T$ that are bad for $\boldsymbol{y}$.  Since $|T|=4$, there is a $t\in T$ such that
\[
\sigma_{t,1}(x_1,j_1),\sigma_{t,2}(x_2,j_2),
\sigma_{t,3}(x_3,j_3)
\]
are pairwise distinct.  Since $\boldsymbol{y}$ was arbitrary, $C$ is a covering family of $[q]^3$ of size 4.

It remains to consider the case $q=6$.  We identify $[6]$ with $\Z_6$, and for each $a\in\Z_6$, let $\tau_a$ be the permutation of $\Z_6$ defined by $\tau_a(x)=x+a$.  Let
\[
B=\{(0,0),(1,2),(2,1),(3,5)\}
\]
and
\[
C=\{(\mathrm{id},\tau_a,\tau_b):(a,b)\in B\}.
\]
We claim that $C$ is a covering family of $[6]^3$.  Suppose $(x_1,x_2,x_3)\in\Z_6^3$.  For $(a,b)\in B$, the elements
\[
x_1,\quad x_2+a,\quad x_3+b
\]
fail to be pairwise distinct only if
\[
x_1=x_2+a,\qquad x_1=x_3+b,\qquad\text{or}\qquad x_2+a=x_3+b.
\]
The first equality can hold for at most one $(a,b)\in B$ since the first coordinates of the elements of $B$ are pairwise distinct.  Similarly, the second equality can hold for at most one $(a,b)\in B$ (since the second coordinates of the elements of $B$ are pairwise distinct).  Finally, the values of $a-b$ over all possible $(a,b) \in B$ are: 0, 5, 1, and 4
which are pairwise distinct in $\Z_6$.  Thus, the third equality can also hold for at most one $(a,b)\in B$.  Consequently, at most three elements of $C$ fail to map $(x_1,x_2,x_3)$ to a triple with pairwise distinct coordinates.  Since $|C|=4$, at least one element of $C$ maps $(x_1,x_2,x_3)$ to a triple with pairwise distinct coordinates.  Thus $C$ is a covering family of $[6]^3$ of size 4.
\end{proof}

We are now ready to prove Theorem~\ref{thm: Aparna}, which we restate.

\begin{customthm}{\ref{thm: Aparna}}
For $k \in \N$ and $k \geq 2$, let $G$ be a graph such that $\chi(G) = \chi_{DP}(G) = k$. Then, $\chi_{DP}(G \square K_{3,t}) = k+3$ whenever $t \geq 4 \left(\frac{P(G,k+2)}{k+2}\right)^3$. 
\end{customthm}

\begin{proof}
Since $k\geq 2$, we have $k+2\geq 4$.  By Lemma~\ref{lem: 3cover}, there exists a covering family $C$ of $[k+2]^3$ with $|C|=4$.  The result now follows immediately from Theorem~\ref{thm: new} with $l=3$.
\end{proof}

We end this section by proving a lemma that, together with Theorems~\ref{thm: new} and~\ref{thm: mucover}, immediately implies Theorem~\ref{thm: letzgoupper}.

\begin{lem}\label{lem: letzgoupper}
Suppose $k,l \in \mathbb{N}$ with $l \geq 2$.  Let
\[
p_{k,l}=\frac{(k+l-1)!}{(k-1)!(k+l-1)^l}
\]
and let $\lambda_{k,l}=-\log(1-p_{k,l})$.  If $k=1$, then there is a covering family of $[k+l-1]^l$ of size at most
\[
\left\lceil \frac{\log((k+l-1)^l\lambda_{k,l}/2)}{\lambda_{k,l}} \right\rceil
+
\left\lfloor \frac{1}{\lambda_{k,l}}+\frac{l}{2}\right\rfloor.
\]
If $k\geq 2$, then there is a covering family of $[k+l-1]^l$ of size at most
\[
\left\lceil \frac{\log((k+l-1)^l\lambda_{k,l}/2)}{\lambda_{k,l}} \right\rceil
+
\left\lfloor \frac{1}{\lambda_{k,l}}+\frac{1}{2}\right\rfloor.
\]
\end{lem}

\begin{proof}
Let $n=k+l-1$, and note that if $(\sigma_1,\ldots,\sigma_l)$ is chosen uniformly at random from $(S_{[n]})^l$, then for any fixed $(a_1,\ldots,a_l)\in[n]^l$, the probability that $|\{\sigma_1(a_1),\ldots,\sigma_l(a_l)\}|=l$ is
\[
p_{k,l}= \frac{\prod_{i=0}^{l-1}((n-i)!(n-1)!/(n-i-1)!)}{(n!)^l} = \frac{n!}{(k-1)!n^l}.
\]
Now, we let $\lambda_{k,l}=-\log(1-p_{k,l})$ and
\[
m= \left\lceil \frac{\log(n^l\lambda_{k,l}/2)}{\lambda_{k,l}} \right\rceil.
\]
Choose $m$ elements of $(S_{[n]})^l$ independently and uniformly at random, and let $R$ be the set of elements of $[n]^l$ not covered by any of our $m$ randomly chosen elements of $(S_{[n]})^l$. By linearity of expectation and the fact that $m$ is defined with a ceiling,
\[
\mathbb{E}(|R|)=n^l(1-p_{k,l})^m=n^le^{-\lambda_{k,l}m} \leq \frac{2}{\lambda_{k,l}}.
\]
So, there is a collection of at most $m$ elements of $(S_{[n]})^l$ such that the number of elements in $[n]^l$ not covered by these elements is at most $\left\lfloor2/\lambda_{k,l}\right\rfloor$. Let $A$ be such a collection, and let $C$ be the set of elements in $[n]^l$ not covered by any element in $A$.

Now, suppose that $k\geq2$. We claim that any two distinct elements $\boldsymbol{x}=(x_1,\ldots,x_l)$ and $\boldsymbol{y}=(y_1,\ldots,y_l)$ of $[n]^l$ can be covered by one element of $(S_{[n]})^l$. Suppose first that $\boldsymbol{x}$ and $\boldsymbol{y}$ differ in at least two coordinates. Choose pairwise distinct $a_1,\ldots,a_l\in[n]$. Let $I=\{i\in[l]:x_i\neq y_i\}$. By cyclically permuting the elements $a_i$ with $i\in I$, we obtain pairwise distinct elements $b_1,\ldots,b_l$ such that $b_i=a_i$ when $i\notin I$ and $b_i\neq a_i$ when $i\in I$. Now we will construct a $(\sigma_1,\ldots,\sigma_l)\in(S_{[n]})^l$. First, for each $i\in[l]$ let $\sigma_i(x_i)=a_i$ and $\sigma_i(y_i)=b_i$. This is possible since $a_i=b_i$ if and only if $x_i=y_i$. Second, arbitrarily extend $\sigma_i$ to a permutation of $[n]$ for each $i\in[l]$. Clearly, $(\sigma_1,\ldots,\sigma_l)$ covers both $\boldsymbol{x}$ and $\boldsymbol{y}$.

Next suppose $\boldsymbol{x}$ and $\boldsymbol{y}$ differ in exactly one coordinate, and suppose that coordinate is $j$. Choose pairwise distinct $a_1,\ldots,a_l\in[l]$, and note that since $k\geq2$, $l+1\in[n]$. Now we will construct a $(\sigma_1,\ldots,\sigma_l)\in(S_{[n]})^l$. First, for each $i\in[l] - \{j\}$ let $\sigma_i(x_i)=\sigma_i(y_i)=a_i$. Second, let $\sigma_j(x_j)=a_j$ and $\sigma_j(y_j)=l+1$. Finally, extend $\sigma_i$ to a permutation of $[n]$ for each $i\in[l]$. It is easy to see that $(\sigma_1,\ldots,\sigma_l)$ covers both $\boldsymbol{x}$ and $\boldsymbol{y}$.

Now, arbitrarily pair the elements of $C$, and note there will be at most one element of $C$ not paired. We know that we can cover the elements of $C$ with at most $\lceil |C|/2\rceil$ elements of $(S_{[n]})^l$. Thus, there is a covering family of $[n]^l$ of size at most $|A|+\lceil |C|/2\rceil$. The desired result for $k\geq2$ follows from
\[
\left\lceil\frac{|C|}{2}\right\rceil \leq \left\lfloor\frac{1}{\lambda_{k,l}}+\frac12\right\rfloor.
\]

Now, suppose that $k=1$ which means $n=l$. Two distinct elements of $[l]^l$ can be covered by one element of $(S_{[l]})^l$ whenever at least two of their coordinates disagree by the same cyclic permutation argument given above for $k\geq2$. On the other hand, two elements of $[l]^l$ that differ in exactly one coordinate cannot be covered by one element of $(S_{[l]})^l$. Indeed, the two resulting $l$-tuples of images would each have to be permutations of $[l]$, while agreeing in $l-1$ coordinates and disagreeing in exactly one coordinate which is impossible.

Let $G$ be the graph with vertex set $C$ with edges defined so that two elements of $C$ are adjacent in $G$ if and only if they differ in at least two coordinates. Furthermore, let $M$ be a maximal matching in this graph. Any vertices in $C$ not saturated by $M$ pairwise differ in exactly one coordinate. In fact, if two of them differ in coordinate $j$, then every other unsaturated vertex must agree with them in every coordinate other than $j$. Therefore there are at most $l$ vertices in $C$ not saturated by $M$. Let $t$ be the number of vertices in $C$ not saturated by $M$, and note that $t$ and $|C|$ have the same parity. So, we can cover the elements of $C$ with at most $(|C|+t)/2$ elements of $(S_{[l]})^l$. Thus, there is a covering family of $[l]^l$ of size at most $|A|+(|C|+t)/2$. The desired result for $k=1$ follows from
\[
\frac{|C|+t}{2}\leq \left\lfloor\frac{|C|+l}{2}\right\rfloor \leq \left\lfloor\frac{1}{\lambda_{k,l}}+\frac{l}{2}\right\rfloor.
\]
\end{proof}

The $k=1$ portion of Theorem~\ref{thm: letzgoupper} now follows from Lemma~\ref{lem: letzgoupper} and Theorem~\ref{thm: mucover}, while the $k\geq2$ portion follows from Lemma~\ref{lem: letzgoupper} and Theorem~\ref{thm: new}.  Our final table summarizes the bounds on $\mu(l)$ obtained in this paper together with the previously known bounds from Theorem~\ref{thm: mudrock}.  The middle column gives the known exact values of $\mu(l)$.

\begin{center}
\begin{tabular}{c|cc|c|cc}
\hline
& \multicolumn{2}{c|}{Lower Bounds} & & \multicolumn{2}{c}{Upper Bounds}\\
$l$ & Thm.~\ref{thm: mudrock} & Thm.~\ref{thm: uptheoddlow} & Exact $\mu(l)$ & Thm.~\ref{thm: letzgoupper} & Thm.~\ref{thm: mudrock}\\
\hline
$2$  & $2$    & --     & $2$  & $3$      & $4$\\
$3$  & $5$    & $5$    & $6$  & $10$     & $13$\\
$4$  & $11$   & $11$   & $12$ & $38$     & $45$\\
$5$  & $27$   & $28$   & ?    & $134$    & $151$\\
$6$  & $65$   & $65$   & ?    & $446$    & $492$\\
$7$  & $164$  & $167$  & ?    & $1443$   & $1557$\\
$8$  & $417$  & $420$  & ?    & $4539$   & $4829$\\
$9$  & $1068$ & $1081$ & ?    & $13993$  & $14736$\\
$10$ & $2756$ & $2762$ & ?    & $42467$  & $44380$\\
$11$ & $7148$ & $7207$ & ?    & $127291$ & $132249$\\
\hline
\end{tabular}
\end{center}
It would be interesting to determine the exact value of $\mu(5)$.  Theorem~\ref{thm: uptheoddlow} gives $\mu(5)\geq28$.  Our computational experiments have produced a covering family of $[5]^5$ of size 46, although we suspect that the actual value of $\mu(5)$ is much closer to 28.  Since we expect that this computational upper bound is far from optimal, we have chosen not to include the corresponding covering family here.

\section*{Acknowledgments}
This paper is partially a result of a visit by the second named author to the fourth named author's institution. The second author's visit was supported by the University of South Alabama. The support of the University of South Alabama is gratefully acknowledged.

\appendix

\section{Verification of the Covering Family for $[4]^4$}
\label{Appendmu4}
We verify that the proposed covering family used in the proof of Proposition~\ref{pro: 4upper} is indeed a covering family. Recall that for each $j\in[12]$, we let
\[
\pi_j=(\sigma_{j,1},\sigma_{j,2},\sigma_{j,3},\sigma_{j,4}),
\]
where the permutations are given in the proof of Proposition~\ref{pro: 4upper}. As in the proof, for each $j\in[12]$, let $A_j$ consist of all $(a_1,a_2,a_3,a_4)\in[4]^4$ such that
\[
\{\sigma_{j,1}(a_1),\sigma_{j,2}(a_2),
\sigma_{j,3}(a_3),\sigma_{j,4}(a_4)\}=[4].
\]
For each $(a_1,a_2,a_3,a_4)\in[4]^4$, we say that the multiplicity of $(a_1,a_2,a_3,a_4)$ is the number of $j \in [12]$ satisfying $(a_1,a_2,a_3,a_4)\in A_j$. The following short computation verifies the multiplicity of each element of $[4]^4$. In fact, it verifies that exactly 240 elements of $[4]^4$ have multiplicity one and the remaining elements have multiplicity three.

\begin{verbatim}
from itertools import product
from collections import Counter

P = [
    [(1,2,3,4),(3,2,4,1),(3,2,1,4),(3,1,4,2)],
    [(1,2,3,4),(1,3,2,4),(3,2,1,4),(1,4,2,3)],
    [(1,2,3,4),(1,3,2,4),(1,4,3,2),(2,4,1,3)],
    [(1,2,3,4),(4,3,2,1),(2,3,4,1),(1,2,3,4)],
    [(1,2,3,4),(2,1,3,4),(3,2,1,4),(2,1,3,4)],
    [(1,2,3,4),(4,2,1,3),(3,2,1,4),(4,2,1,3)],
    [(1,2,3,4),(4,2,1,3),(1,4,3,2),(1,2,4,3)],
    [(1,2,3,4),(2,1,3,4),(1,4,3,2),(3,1,2,4)],
    [(1,2,3,4),(2,1,4,3),(2,1,4,3),(1,4,3,2)],
    [(1,2,3,4),(4,1,2,3),(4,3,2,1),(3,4,1,2)],
    [(1,2,3,4),(2,3,4,1),(4,1,2,3),(3,2,1,4)],
    [(1,2,3,4),(3,2,4,1),(1,4,3,2),(4,1,3,2)]
]

multiplicity = []

for a in product(range(1,5), repeat=4):
    count = 0
    for row in P:
        image = {row[i][a[i]-1] for i in range(4)}
        if len(image) == 4:
            count += 1
    multiplicity.append(count)

dist = Counter(multiplicity)

print("Multiplicity distribution:")
for m in sorted(dist):
    print("multiplicity", m, ":", dist[m], "tuples")

print("Total tuples:", len(multiplicity))
print("Minimum multiplicity:", min(multiplicity))
print("Maximum multiplicity:", max(multiplicity))
\end{verbatim}

The output is

\begin{verbatim}
Multiplicity distribution:
multiplicity 1 : 240 tuples
multiplicity 3 : 16 tuples
Total tuples: 256
Minimum multiplicity: 1
Maximum multiplicity: 3
\end{verbatim}

Thus, every element of $[4]^4$ belongs to at least one of the sets $A_j$.

\bibliographystyle{hplain}
\bibliography{bibliography}

\section*{Statements and Declarations}

\paragraph{Funding.}
The second named author's visit to the fourth named author's institution
was supported by the University of South Alabama.

\paragraph{Competing interests.}
The authors have no relevant financial or non-financial interests to disclose.

\paragraph{Data availability.}
No datasets were generated or analyzed during the current study.

\paragraph{Code availability.}
The code and solver output used for the integer programming computation
in the proof of Proposition~\ref{pro: 4upper} are publicly available in
\cite{M26comp}.

\paragraph{Declaration of generative AI and AI-assisted technologies in
the manuscript preparation process.}
During the preparation of this manuscript, the authors used ChatGPT
(OpenAI) and Claude (Anthropic) for limited language editing, assistance
in generating and reviewing code, and preliminary discussions of proof
ideas. All AI-assisted material was critically reviewed, revised, and
independently verified by the authors. The authors take full
responsibility for the mathematical arguments, computational results,
and final content of the manuscript.

\end{document}